\documentclass[a4paper,12pt]{article}
\usepackage[utf8]{inputenc}

\usepackage{graphicx, amsfonts, amsmath, amssymb, amsthm, color}
\usepackage[colorlinks=true, allcolors=blue]{hyperref}
\newcommand{\pr}{{\mathbb P}}
\newcommand{\ep}{{\mathbb E}}
\newcommand{\ZZ}{{\mathbb Z}}
\newcommand{\II}{{\mathbb I}}
\newcommand{\err}{{\rm err}}
\newcommand{\vc}[1]{{\mathbf{#1}}}
\newcommand{\ol}[1]{{\overline{#1}}}
\newcommand{\KK}{{\mathcal{K}}}
\newcommand{\KKCOMP}{\widehat{\KK}_{\rm COMP}}
\newcommand{\cov}{{\rm Cov\;}}
\newcommand{\var}{{\rm Var\;}}
\newcommand{\bin}{{\rm Bin\;}}
\newcommand{\po}{{\rm Po\;}}
\newcommand{\fm}[2]{M_{(#1)}\left( #2 \right)}

  \newtheorem{theorem}{Theorem}[section]
  \newtheorem{corollary}[theorem]{Corollary}
  \newtheorem{proposition}[theorem]{Proposition}
  \newtheorem{lemma}[theorem]{Lemma}
  \newtheorem{definition}[theorem]{Definition}
  
\newtheorem{remark}[theorem]{Remark}

\title{A negative binomial approximation in group testing}
\author{Letian Yu\thanks{Department of System Engineering and Engineering Management, The Chinese University of Hong Kong, Hong Kong SAR. Email: {\tt{Letian.Yu@link.cuhk.edu.hk}}} \and Fraser Daly\thanks{Department of Actuarial Mathematics and Statistics, Heriot--Watt University, Edinburgh, EH14 4AS, UK. Email: {\tt{F.Daly@hw.ac.uk}}} \and Oliver 
Johnson\thanks{School of Mathematics, University of Bristol, Fry Building, Woodland Road, Bristol, BS8 1UG, UK.
Email: {\tt O.Johnson@bristol.ac.uk}
}}
\date{\today}

\begin{document}

\maketitle
\begin{abstract}
\noindent
We consider the problem of group testing (pooled testing), first introduced by Dorfman. For non-adaptive testing strategies, we refer to a non-defective item as `intruding' if it only appears in positive tests. Such items cause mis-classification errors in the well-known COMP algorithm, and can make other algorithms produce an error. It is therefore of interest to understand the distribution of the number of intruding items.  We show that, under Bernoulli matrix designs, this distribution is well approximated in a variety of senses by a negative binomial distribution, allowing us to understand the performance of the two-stage conservative group testing algorithm of Aldridge.
\end{abstract}


\vspace{20pt}


\section{Introduction to group testing}

The group testing (pooled testing) problem was introduced by Dorfman \cite{dorfman}, and provides a way of efficiently finding a small number of infected individuals in a large population. The key construction underlying this method is a so-called pooled test: given a subset $S$ of the population, we combine samples from each member of $S$ into a testing pool, and test them all together. We suppose that such a test returns a positive result if and only if at least one person in $S$ is infected. Hence, a negative test allows us to deduce that every person in $S$ is not infected, allowing for efficient screening of individuals. Some inference can be drawn from a positive test, but the analysis is typically more involved. 

By carrying out a series of pooled tests, we hope to efficiently identify all the infected individuals \textit{--} using as few tests as possible, and ideally using simple algorithms which do not require intensive computation. 

The group testing problem has generated an extensive literature, surveyed for example in \cite{johnson49,du}, and a large variety of variations on the problem exist. Group testing has been proposed as a solution to problems in a wide variety fields (see \cite[Section 1.7]{johnson49} for a survey of some of these applications). In engineering and computer science in particular, it has been used to efficiently search computer memories \cite{kautz}, to give a novel data compression algorithm \cite{hong}, to detect high-demand items in databases \cite{cormode}, to bound the performance of multi-access communications channels \cite{wolf}, and in many other problems besides. Due to the shortage of tests early in the COVID-19 pandemic, group testing was a natural solution proposed to find infected individuals, and has indeed been deployed at scale in a variety of countries (see \cite[Section 6]{aldridge7} for a review of some such uses).

One key distinction is whether we are allowed to employ adaptive testing strategies (where the choice of individuals to be tested can depend on the outcome of previous tests) or are restricted to non-adaptive ones (where the test design is chosen in advance). Clearly, the ability to search adaptively cannot harm us, and indeed it is known (see \cite[Section 1.5]{johnson49}) that in this case Hwang's algorithm \cite{hwang}, based on efficient binary search, requires a number of tests which is asymptotically optimal for a wide range of parameters. However, in many circumstances, adaptive algorithms may not be an option, and there is independent mathematical interest in understanding the performance of algorithms which are non-adaptive or are restricted to a small number of stages. 

One motivating example for the study of group testing algorithms with restricted stages  is that of COVID testing. It has been known since the early days of the coronavirus pandemic  that  a sample from one infected individual gives a strong enough  PCR positivity signal  when mixed in a pool of 32 or 64 samples that group testing is a potentially viable strategy \cite{yelin}. However, to take advantage of the ability of PCR machines to perform 96 or more tests at a time in parallel \cite{erlich2} we need to use non-adaptive algorithms. Further, as argued in for example \cite{mutesa2020}, if each round of tests takes a few hours to be processed, then multi-stage binary search algorithms can give information about the infection status of individuals too late to be useful, meaning that the virus could have already been passed on before the results are received. For this reason, in this article we will focus on non-adaptive and two-stage algorithms.

When we work non-adaptively, we must declare our testing strategy in advance, and it is perhaps not immediately obvious which strategy to choose. One simple idea, which we will refer to as Bernoulli testing,  consists of randomly placing each member of the population into each pool with the same probability $p$, each such choice being taken independently of one another. In fact, if there are $k$ infected individuals, it is generally a good strategy to choose $p=1/k$, so that there is on average one infected individual in each test pool. Of course, Bernoulli testing is not the only strategy, and improved performance can be obtained by placing each person in a fixed number of tests at random \cite{johnson42}, or by more advanced test designs \cite{coja}. However, Bernoulli testing is simple to describe and analyse, and generally gives performance \cite[Section 2]{johnson49} within a constant multiple of the best possible, so we will focus on it here.

Having chosen a particular non-adaptive test strategy, a key question is how we find the infected individuals. A variety of algorithms are possible but one particularly simple one is referred to as COMP after its use in the paper \cite{chan}, but dates back at least to the work of Kautz and Singleton \cite{kautz}. This algorithm works as follows: as mentioned above, each person who appears in a negative test is guaranteed to be not infected. Hence, we can build a list of non-infected people by collecting together the people from each negative test. For definiteness, we simply assume that everyone else is infected.

Given enough tests, each non-infected person should appear in at least one negative test, but using arguments based on the coupon-collector problem we can deduce that this may require more tests than we would otherwise hope. If we use insufficiently many tests, the algorithm is likely to fail, with a clear single source of error. That is, if some non-infected person only appears in positive tests, then they will be incorrectly classified as infected. We refer to such a person as `intruding', and the focus of this paper will be to count the number of intruding individuals, which we will refer to as $G$. (Each person declared to be non-infected will definitely be so -- see \cite[Lemma 2.3]{johnson49}.) 

The COMP algorithm succeeds in deducing every individual's infection status exactly if and only if $G=0$, but by understanding the distribution of $G$ we can also consider some related issues. First, as mentioned previously, given perfect testing, COMP never classifies a non-infected person as infected, and can be used to provide a quick screening of the population. Knowledge of $G$ tells us precisely how many healthy people would be wrongly quarantined as a result of this screening, so given a particular tolerance of this effect we could choose the number of tests accordingly.

Second, we can regard COMP as the first stage of a `conservative two-stage group testing' algorithm as described by Aldridge \cite{aldridge6}, where following an initial screening using COMP we choose to test each person who has not received a clean bill of health using individual testing. \cite[Theorem 1]{aldridge6} describes the expected number of tests for such a procedure to succeed. If there are $k$ infected people, clearly the second stage requires $G+k$ tests to succeed, so by understanding the distribution of $G$ we can approximate the probability this two stage algorithm will succeed, which can give more information than the expected value. Further, given an overall budget of $T$ tests, we might wish to know the optimal number of tests $T_1$ to use for the initial COMP stage, and analysing the distribution of $G$ will give insight into this.

Finally, the purely non-adaptive DD algorithm introduced in \cite{johnson33} uses COMP as a first stage of the analysis, and performs a further analysis based on looking for positive tests that contain exactly one `non-screened' item. Informally, we know that DD will succeed if the number of intruding items $G$ is much less than the number of defectives $k$, so again by understanding the distribution of $G$ we gain insight into the performance of DD.

The structure of the remainder of the paper is as follows. In Section \ref{sec:notation} we give a more formal introduction to the group testing problem, including introducing notation, defining the COMP algorithm and proving some simple properties of the number of intruding items $G$. In Section \ref{sec:main} we show that $G$ can be well approximated by a negative binomial distribution, first by considering a limiting argument that shows convergence of all falling moments in an asymptotic limit and then giving a more detailed bound based on a novel adaptation of the Stein--Chen method which gives bounds in finite blocklength settings as well. Section \ref{sec:implications} discusses the implications of these results for various group testing algorithms, before a brief conclusion is given in Section \ref{sec:conc}. 

\section{Notation and definitions} \label{sec:notation}
\subsection{Group testing setup}

We now state the group testing problem in slightly more formal language and introduce some notation similar to that of \cite{johnson49}. We will write $n$ for the total population of individuals, which we will refer to as `items'. Instead of referring to infected and healthy individuals, we will follow standard group testing terminology by calling them `defective' and `non-defective' respectively. We  write $\KK$ for the set of defective items (or defective set for short) and $k = | \KK |$ for the total number of defective items, and $T$ for the number of tests.

As is standard, we can represent a non-adaptive testing strategy by a binary $T \times n$ test matrix $X$, with rows corresponding to tests and columns corresponding to items. Here the entry $X_{ti} = 1$ means that item $i$ appears in test $t$. In this paper we focus on Bernoulli testing, where the $(X_{ti})$ are independent Bernoulli random variables with parameter $p$. We will refer to this as a `Bernoulli test design with parameter $p$', and this design will apply throughout. Of particular interest will be the case $p = 1/k$.

The outcome of test $t$ is represented as a binary value $Y_t$  (where $Y_t = 1$ means a positive test) which can be calculated for this test matrix as 
\begin{equation} \label{eq:grouptestresult}  Y_t = \bigvee_{i \in \KK} X_{ti}, \end{equation}
where $\bigvee$ represents a standard binary OR, capturing the fact that each test is positive if and only if it contains at least one of the items in $\KK$.

As in \cite{johnson33} we write $q_0 = (1-p)^k$, noting that each test is positive independently with probability $1-q_0$ (since it is negative if and only if it contains none of the $k$ defectives). 

Again, as in \cite{johnson33}, \cite{coja} and other papers we will often study what is referred to in \cite{johnson49} as the {\em sparse regime}. In this setting the number of items $n$ tends to infinity and the number of defectives $k = k(n) = n^\theta$ for some explicit parameter $\theta\in(0,1)$. In this context it is natural to consider an asymptotic regime where the number of tests $T= (c/q_0) k \log(n)$ for some constant $c$,  noting that if $p=1/k$ then asymptotically $q_0$ converges to $e^{-1}$. Here and throughout our work we write $\log$ for the natural logarithm.

Another asymptotic setting of interest (see for example \cite{aldridge5} and \cite[Section 5.5]{johnson49}) is referred to as the {\em linear regime}. Here, again $n$ tends to infinity, and the number of defectives $k = k(n) = \beta n $ for some explicit parameter $\beta \in (0,1)$. In this context, we consider an asymptotic regime where the number of tests $T=c n$ for some constant $c$. Although in this regime Aldridge \cite{aldridge5} proved that no non-adaptive algorithm can outperform individual testing, there is still interest in understanding the performance of two-stage or adaptive algorithms, and analysis of $G$ of the kind presented in this paper can help with this.

However, in many practical group testing contexts we are also interested in what we refer to as the {\em finite blocklength setting}, following terminology popularised for example by the work of Polyanskiy {\em et al.} \cite{polyanskiy2}. In this context we wish to understand the performance of algorithms in solving concrete problems such as $n=500$, $k=10$ (see \cite{johnson33}), where asymptotic bounds may not necessarily give the best guide to actual performance. Interest in finite blocklength problems was particularly prompted by the COVID pandemic, where for example the use of 96-well PCR plates \cite{erlich2} means that the number of tests may be bounded by (a multiple of) 96. This setting has typically been less well explored than asymptotic settings such as the sparse or linear regimes described above, however we provide some bounds in this context.  

At various points in our analysis we will find it useful to work in terms of the falling moments of various random variables, and  we will write $\fm{s}{Y} := \ep \left( Y \right)_{(s)} = \ep Y(Y-1) \ldots (Y-s+1) = \ep Y!/(Y-s)!$ for the $s$th falling moment of random variable $Y$. %For example, $\fm{2}{Y} = \ep Y(Y-1)$, so we can express $\var(Y) = \fm{2}{Y} + \fm{1}{Y} - \fm{1}{Y}^2$. 
We will also write $w(\vc{u})$ for the Hamming weight of the binary vector $\vc{u}$.

\subsection{COMP algorithm}

The COMP algorithm uses the test outputs $Y$ and matrix $X$ to produce an estimate of the defective set $\KK$ which we will write as $\KKCOMP$. In fact it is easier to consider the complement of $\KKCOMP$: an item will appear in this complement if it appears in some negative test. Formally speaking
\begin{equation} \label{eq:COMPdeduce}
 \KKCOMP^c = \left\{  \mbox{$i$: $Y_s = 0$ for some $s$ with $X_{si} = 1$} \right\}. \end{equation}
Notice that if item $j$ really is defective (i.e. $j \in \KK$) then for every test $s$ with $X_{sj} = 1$ then $Y_s = 1$ (by the definition of the group testing action in \eqref{eq:grouptestresult}), so that \eqref{eq:COMPdeduce} implies that $j$ is not in $\KKCOMP$. In other words $\KK \subseteq \KKCOMP$ (see \cite[Lemma 2.3]{johnson49}).

COMP is an attractive algorithm because it is simple to perform and interpret, and because of this performance guarantee in one direction. However, in practice if we don't perform enough tests then $\KKCOMP$ can be significantly larger than $\KK$, meaning that many non-defective items would be misclassified, potentially causing problems in healthcare related situations where unnecessary quarantine could result.

For this reason, Aldridge \cite{aldridge6} proposed what he refers to as a conservative two-stage algorithm. Here, given a total budget of $T$ tests, we should use $T_1$ of them to perform the COMP algorithm in the usual way, and then the remaining $T_2 := T-T_1$ tests to perform individual testing of each of the items in $\KKCOMP$, which have not been classified as non-defective. While this algorithm may be inferior in performance to a two-stage algorithm which uses the $T_1$ tests to perform the DD algorithm \cite{johnson33} followed by individual testing, it has the advantage of being transparent to perform for healthcare professionals without a mathematical background.

Clearly the conservative two-stage algorithm of Aldridge \cite{aldridge6} will succeed in finding all the defective items if the number of second stage tests  is greater than or equal to the number of items to be tested. That is, it will succeed when $T_2 = T-T_1 \geq |\KKCOMP|$, meaning that we would like to find the size of $\KKCOMP$. Further, for a given budget of $T$ tests, since larger values of $T_1$ give smaller  $\KKCOMP$ (more stage one tests allow more non-defective items to be screened out) but leave fewer tests available in the second stage, we would like to find a sensible choice of $T_1$ that manages this tradeoff.

\subsection{Basic properties of intruding items}

We now define the key property we will study in this paper:
\begin{definition} \label{def:intruding}
We define
\begin{enumerate}
\item 
a non-defective item as `intruding' if it only appears in positive tests.
\item
the binary random variables 
$$ G_i =\II \left(\mbox{item $i$ is non-defective and intruding} \right),$$
$\vc{G} = (G_1, \ldots, G_n)$ the binary vector with these components, and $G = \sum_{i} G_i = w(\vc{G})$.
\end{enumerate}
\end{definition}

More formally, if non-defective item $\ell$ is intruding then $Y_s = 1$ for every $s$ with $X_{s\ell}=1$, so that (by \eqref{eq:COMPdeduce})  we know $\ell \notin \KKCOMP^c$, so $\ell \in \KKCOMP$. In other words, if a non-defective item is intruding then COMP will mistakenly declare it to be defective. Since non-intruding items are declared to be non-defective by COMP, we know that the size of $\KKCOMP$ (which determines the success of the two-stage algorithm of Aldridge \cite{aldridge6}) is exactly $|\KKCOMP| = k + G$. 

For a given non-defective item $i$ we can work out the marginal distribution of $G_i$ relatively easily:
\begin{lemma} \label{lem:gimarginal} Under a Bernoulli test design with parameter $p$, for each non-defective item $i$, the marginal distribution of $G_i$ is Bernoulli with parameter $(1-p q_0)^T$, where we recall that we write $q_0 = (1-p)^k$.
\end{lemma}
\begin{proof}
Item $i$ is intruding if no test contains item $i$ and no  defective item. The probability of the event that test $t$ contains item $i$ and no defective item is $p(1-p)^k = p q_0$, so since successive tests are independent, the chance that we avoid this event for each test is $(1-p q_0)^T$. 
\end{proof}

\section{Main results} \label{sec:main}

\subsection{Moments and associated random variables}

If the $G_i$ were independent, then the analysis of the distribution of $G$ would be easy: using Lemma \ref{lem:gimarginal} then $G$ would be binomial with parameters $n-k$ and $(1-p q_0)^T$. However, there is a dependence between the $G_i$. In fact, if we learn that a given item is intruding, that suggests there might be more positive tests than average, which would imply that other items are more likely to be intruding.

We formalise this intuition by showing that the $G_i$ are more likely be equal to 1 together than independence would imply. That means that if COMP fails, it is more likely to fail badly (with a large total $G$) than a naive analysis based on Lemma \ref{lem:gimarginal} might suggest. We prove this in two ways, first by showing in Corollary \ref{cor:poscorr} that the $G_i$ are pairwise positively correlated, and second by proving a stronger result (Proposition \ref{prop:ass}) which shows that the $G_i$ have the property of association (see Definition \ref{def:ass}), which is stronger than positive correlation (see the discussion in \cite{esary} for example).

In fact, we will deduce the  positive correlation result (Corollary \ref{cor:poscorr}) from an expression  for the falling moments of $G$ (Proposition \ref{prop:fmg}), which may be of independent interest, extending the result for $\ep G$ implicit in \cite[Theorem 1]{aldridge6}). We describe the distribution of $G$ as a binomial mixture of binomial distributions as follows. As in \cite{johnson33},  if we 
let $M_{0}$ be the number of negative tests then we know that:
\begin{align}
M_{0} & \sim \bin\left(T,(1-p)^{k}\right),  \label{eq:M0def} \\ 
G \mid {M_{0}=m} & \sim \bin \left(n-k,(1-p)^{m}\right). \label{eq:GM0def}
\end{align}
The first result follows because a test is negative if and only
if it contains no defective items, and for Bernoulli testing this occurs
independently across tests with probability $q_0 = (1-p)^{k}$. Moreover, the second result follows because a non-defective item is intruding if and only if it does not appear in any of the $M_0$ negative tests, and each non-defective item is present in a given test independently with probability $p$. We can use this to prove the following result:
\begin{proposition} \label{prop:fmg}
Under a Bernoulli test design with parameter $p$, 
 the falling moments of $G$ are given by
\begin{equation} \label{eq:fmg}
\fm{s}{G} =  \ep G(G-1) \ldots (G-s+1) = \binom{n-k}{s} s!  \left( 1 - q_0 \left( 1- (1-p)^s \right) \right)^T,
\end{equation}
for any integer $s \geq 0$.
\end{proposition}

We prove this result using the following intermediate lemma:

%\begin{proof} See Appendix \ref{app:prooffmg}. \end{proof}

%\section{Proof of Proposition \ref{prop:fmg}} \label{app:prooffmg}%
%We first recall the following result, which we state without proof.

\begin{lemma}[Falling Moments of Binomial Distribution] 
\label{lem:fmbin} Suppose $X \sim \bin(L,t)$. Then the $s$th falling moment of $X$ is given by:
\begin{equation}
\fm{s}{X} = \binom{L}{s} s! \cdot t^{s}.
\end{equation}
\end{lemma}
%\begin{proof}
%We directly evaluate the falling moments as
%\begin{align*}
%\ep(X)_{(s)} &= \sum_{x=0}^{L} \binom{L}{x} t^{x}(1-t)^{L-x} \frac{x !}{(x-s) !} \\ 
%&= \sum _{x=s}^L \frac{L !}{(L-x) ! (x-s) !} t^{x}(1-t)^{L-x} \\ 
%&=  \frac{L !}{(L-s) !} t^{s} \cdot \sum_{x=s}^{L} \binom{L-s}{x}  t^{x-s}(1-t)^{L-x} 
%=  \frac{L !}{(L-s) !} t^{s} \II(L-s \geq 0), 
%\end{align*}
%as required.
%\end{proof}

Proposition \ref{prop:fmg} follows using Lemma \ref{lem:fmbin} when we recall the distributions of $M_0$ and $G \mid {M_0 = m}$ from \eqref{eq:M0def} and \eqref{eq:GM0def}, and apply the law of iterated expectation to express $\fm{s}{G}  = \ep \left[ \ep \left( (G)_{s} | {M_0} \right) \right]$. We omit the details for brevity.

Note that we can give an alternative proof of Proposition \ref{prop:fmg}, using \cite[Lemma 2.2]{johnson26}, which gives a multinomial-type expansion for falling factorials based on the Vandermonde identity. This expansion simplifies in the case of binary random variables to give the fact that the falling moment can be expressed as a sum over sets:
\begin{equation}
\fm{s}{G} = s! \sum_{S: |S| = s} \ep \left( \prod_{i \in S}
G_{i} \right).
\end{equation}
The summation over sets contributes $s! \binom{n-k}{s}$ equal expectation terms, each one of which equals the probability that all elements of a specified set are intruding, which corresponds to the event that none of them ever appear in a negative test, so the result follows by independence of all test items.

Using Proposition \ref{prop:fmg} we can deduce the following result that shows that the $G_i$ have positive pairwise correlation:

\begin{corollary} \label{cor:poscorr}
For non-defective items $i \neq j$:
\begin{equation} \label{eq:poscov} \cov(G_i, G_j) = 
\left( 1 - q_0 \left( 2p - p^2 \right) \right)^T
- \left( 1 - q_0 p \right)^{2T}
\geq 0.\end{equation}
\end{corollary}
\begin{proof}
We consider the variance of $G$ in two different ways:
\begin{equation} \label{eq:var1}
\var(G)  %=  \ep G(G-1) + \ep G - (\ep G)^2 
= \fm{2}{G} + \ep G - (\ep G)^2,
\end{equation}
and (using the symmetry between pairs of $G_i$ implied by the Bernoulli matrix design)
\begin{align} 
\var(G) & = \sum_{i} \var(G_i) + \sum_{i \neq j} \cov(G_i,G_j) \nonumber \\
& =  (n-k) \left( \ep G_i - (\ep G_i)^2 \right) + (n-k)(n-k-1) \cov(G_i, G_j) \nonumber \\
& =  \ep G -  \frac{(\ep G)^2}{n-k} + (n-k)(n-k-1) \cov(G_i, G_j), \label{eq:var2}
\end{align}
using the fact that for a binary random variable $Y$ we have $\var(Y) = \ep Y^2 - (\ep Y)^2 = \ep Y - (\ep Y)^2$ and that $(n-k) \ep G_i = \ep G$. Now, equating \eqref{eq:var1} and \eqref{eq:var2} we obtain that
\begin{align*}
\lefteqn{ (n-k)(n-k-1) \cov(G_i, G_j) } \\
& = \fm{2}{G} - (\ep G)^2 \left(1- \frac{1}{n-k} \right) \\
& = (n-k)(n-k-1) \left[ \left( 1 - q_0 \left( 1- (1-p)^2 \right) \right)^T
- \left( 1 - q_0 \left( 1- (1-p) \right) \right)^{2T} \right],
\end{align*}
using the expressions for $\fm{2}{G}$ and $\fm{1}{G}$ from Proposition \ref{prop:fmg}, and the result follows on cancellation.
\end{proof}

However, we can prove a stronger property than just positive correlation. Recall the following definition:
\begin{definition}[\cite{esary}] \label{def:ass} Random variables $\vc{X} = (X_1, X_2, \ldots, X_n)$ are (positively) {\em associated} if, for all increasing functions $f$ and $g$,
\begin{equation} \label{eq:ass} \ep \left( f(\vc{X}) g(\vc{X}) \right)
\geq \ep \left( f(\vc{X})  \right) \ep \left(  g(\vc{X}) \right). \end{equation}
\end{definition}

We will prove the following proposition:
\begin{proposition} \label{prop:ass} Under a Bernoulli test design, the random variables $\vc{G} = (G_1, G_2, \ldots, G_n)$ are associated.
\end{proposition}
\begin{proof} See Appendix \ref{app:proofass}. \end{proof}

Combining Proposition \ref{prop:ass} with Theorem 3.1 of \cite{denuit}, we find that $G$ is larger, in a convex sense, than a binomial random variable $H$ with parameters $n-k$ and $(1-pq_0)^T$. That is,  we have that $\ep g(G)\geq \ep g(H)$ for all real-valued functions $g$ with $g(x+1)-2g(x)+g(x-1)\geq0$ for all positive integers $x$, and for which the expectations exist (where we note from Property 3.4 of \cite{DLU} that convex ordering on the integers is equivalent to convex ordering on the real line). This formalises the notion that $G$ is more variable than it would be if the $G_i$ were independent. 

The function $g(x)=x!/(x-s)!$ satisfies this convexity condition, since direct calculation gives that $g(x+1)-2g(x)+g(x-1) = s(s-1) (x-1) \ldots (x-s+2)$ in this case. We thus deduce that (see Lemma \ref{lem:fmbin} below for the value of $\fm{s}{H})$
\begin{equation} \label{eq:momlowbd}
\fm{s}{G}\geq\fm{s}{H}=\binom{n-k}{s}s!(1-pq_0)^{sT},
\end{equation}
i.e., the falling moments of $H$ act as lower bounds for those of $G$. Indeed, direct calculation gives that the ratio
\begin{equation}
    \frac{\fm{s}{G}}{\fm{s}{H}} = \left( \frac{ 1- q_0 (1- (1-p)^s)}{(1-p q_0)^s} \right)^T =: R(s)^T,
\end{equation}
where $R(s+1) - R(s) = (p q_0 (1-q_0)) \left( 1 - (1-p)^s
\right) (1-p q_0)^{-s-1} \geq 0$, so that the ratio between successive falling moments of $G$ and $H$ is increasing in $s$.

Some numerical illustration of this is given in Table \ref{fig:momcomp} below, along with comparison of falling moments of $G$ with those of other distributions which are more suitable than $H$ as approximations of $G$ and which we now discuss in more detail.

\subsection{Negative binomial approximation}\label{sec:nbapprox}

In this subsection, we start to show that the distribution of $G$ can be approximated by $Z$, where $Z \sim \mbox{NB}(r,q)$ follows the negative binomial distribution. For concreteness, we use the parameterization where the probability mass function for the negative binomial distribution is 
\begin{equation} \label{eq:negbinpar} 
\pr(Z=z) = f(z; r, q) := \frac{\Gamma(z+r)}{\Gamma(r) z !} q^{r}(1-q)^{z} \mbox{\;\; for $z=0,1,2 \ldots$}.\end{equation}
Note that in the case of integer $r$, the normalization constant
$\frac{\Gamma(z+r)}{\Gamma(r) z !} = \binom{z+r-1}{z}$, and $Z$ can be interpreted in terms of the number of failures to see the $r$th success in a sequence of Bernoulli trials. However, we do not require $r$ to be an integer here. 

The parameter $r$ is sometimes referred to as the dispersion. In this sense it is worth noting that the case $r=1$ corresponds to a geometric random variable (as mentioned above) and the limiting regime $r \rightarrow \infty$ and $q = r/(r+\lambda)$ (which corresponds to a mean-preserving limit -- see Lemma \ref{lem:fmnb} below) gives the mass function of a Poisson random variable with mean $\lambda$. However, in the finite blocklength setting we will see that $G$ is often well approximated by a distribution with $1 \ll r \ll \infty$, meaning that neither the geometric nor Poisson approximation are valuable.

One natural question is that of which negative binomial distribution (which choice of parameters) to use. We argue that one natural choice is based on a standard moment matching argument -- since we need two parameters, we need to set 
$\ep G = \ep Z$ and $\ep G^2 = \ep Z^2$. In fact, equivalently, since Proposition \ref{prop:fmg} and 
Lemma \ref{lem:fmnb} below give closed form expressions for the falling moments of $G$ and $Z$, it is 
actually easier to solve $\fm{s}{G} = \fm{s}{Z}$ for $s=1,2$. It is a straightforward exercise involving the Gamma function to prove the following:

\begin{lemma} \label{lem:fmnb}
 (Falling moments of negative binomial distribution) The falling moments of $Z \sim \mbox{NB}(r,q)$ are given by: 
\begin{equation} \label{eq:fmnb}
\fm{s}{Z} = \frac{\Gamma(s+r)}{\Gamma(r)}\left(\frac{1-q}{q}\right)^{s}.
\end{equation}
\end{lemma}
%\begin{proof}
%Using the probability mass function given by \eqref{eq:negbinpar} we obtain
%\begin{align*}
% \ep(Z)_{(s)} &=  \sum_{z=s}^{\infty} \frac{z!}{(z-s) !} \cdot \frac{\Gamma(z+r)}{\Gamma(r) z !} q^{r}(1-q)^{z} \\  &=  \frac{\Gamma(s+r)}{\Gamma(r)}\left(\frac{1-q}{q}\right)^{s} \sum_{z=s}^{\infty} \frac{\Gamma(z+r)}{\Gamma(s+r)(z-s) !} \cdot q^{s+r}(1-q)^{z-s} \\ 
% &= \frac{\Gamma(s+r)}{\Gamma(r)}\left(\frac{1-q}{q}\right)^{s}
% \sum_{z=s}^\infty f(z-s; r+s, q), 
%\end{align*}
%and the result follows.
%\end{proof}

Hence, combining Proposition \ref{prop:fmg} and Lemma \ref{lem:fmnb}, we have successfully matched the moments if
\begin{align}
\frac{r (1-q)}{q} & =  (n-k) \left( 1 - q_0 p \right)^T = \fm{1}{G}, \label{eq:matchmom1} \\
\frac{r (r+1) (1-q)^2}{q^2} & =  (n-k)(n-k-1)  \left( 1 - q_0 \left( 2p-p^2 \right) \right)^T
= \fm{2}{G},\;\;
 \end{align}
or, equivalently, if 
\begin{equation} \label{eq:paramest}
r= \frac{ \fm{1}{G}^2}{ \fm{2}{G} - \fm{1}{G}^2} \mbox{ \;\;\; and \;\;\; }
q = \frac{ \fm{1}{G}}{ \fm{2}{G} + \fm{1}{G} - \fm{1}{G}^2}. 
\end{equation}

In Figure \ref{figure2}, we illustrate the quality of this negative binomial approximation for $G$ in the case $n=500$, $k=10$ and $p=0.1$. We plot the mass function of the negative binomial random variable $Z$ (with parameter choices as in \eqref{eq:paramest}) and the corresponding estimated mass function for $G$ obtained by simulation, for several values of the number of tests $T$ in a non-adaptive algorithm. These experiments show that this negative binomial distribution gives a good approximation to the distribution of $G$ for various different values of $T$.

    % Note: May delete this figure if we don't need individual comparisons
    \begin{figure}[!ht] 
    \centering
    \includegraphics[width=0.95\textwidth]{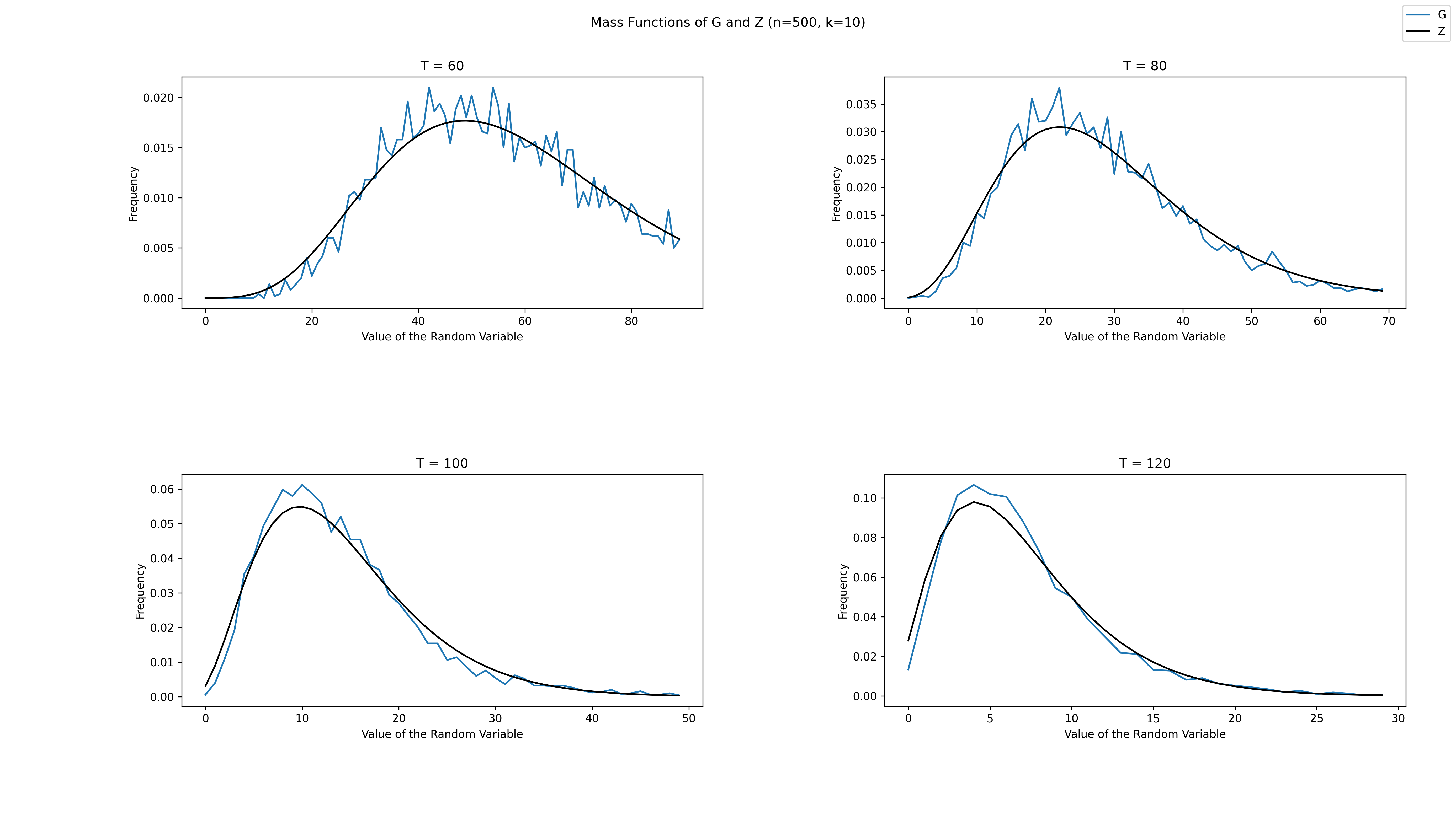} 
    \caption{Probability mass functions of $G$ (blue, obtained by simulation) and approximating negative binomial random variable $Z$ (black). In this case the group testing parameters are $n=500$, $k=10$ and $p=0.1$. The four plots correspond to $T=60,80,100,120$.} 
    \label{figure2} 
    \end{figure}

We also find that for many finite blocklength examples the extra flexibility offered by a two parameter approximation means that the negative binomial distribution with parameters given by \eqref{eq:paramest} approximates the low order falling moments of $G$ (and hence the variance, skewness and kurtosis) better than either the Poisson or geometric distributions with a single parameter chosen by matching means. This is illustrated in Table \ref{fig:momcomp}, again in the case $n=500$, $k=10$, $p=0.1$, $T=100$. In this setting the dispersion parameter $r$ of $Z$ is $3.66$, which is well separated from both $r=\infty$, corresponding to the Poisson, and $r=1$, corresponding to the geometric.

\begin{table}
    \begin{center}
\begin{tabular}{|l||c|c|c|c|c|}
\hline
$s$ & $\fm{s}{G}$ & $\fm{s}{Z}$ & $\fm{s}{Y}$ & $\fm{s}{X}$ & $\fm{s}{H}$ \\
 & (true) & (negative binomial) & (Poisson) & (geometric) & (binomial) \\
 \hline
 \hline
 1   &  14.088 & 14.088 & 14.088 & 14.088 & 14.088\\
 2   & 252.71 & 252.71  & 198.49  & 397.0 & 198.09\\
 3 & 5,716.9 & 5,505.1 & 2,796.6 & 16,779.4 & 2,779.5\\
 4 & 161,487 & 141,110 & 39,400 & 945,605 & 38,919\\
 \hline
\end{tabular}
\caption{First four falling moments of approximating distributions in the case  $n=500$, $k=10$, $p=0.1$, $T=100$. True falling moments $\fm{s}{G}$ given by \eqref{eq:fmg}. Negative binomial falling moments $\fm{s}{Z}$ given by \eqref{eq:fmnb} with parameters given by \eqref{eq:paramest}. Poisson falling moments are given by $\fm{s}{Y} = \lambda^s$, where $\lambda = \fm{1}{G}$ is chosen to match the first moment of $G$. Geometric falling moments are given by $\fm{s}{X} = s! (1/\alpha - 1)^s$ where $\alpha = 1/(1+\fm{1}{G})$ is chosen to match the first moment of $G$. The binomial random variable $H$ and its falling moments are as given in the discussion following Proposition \ref{prop:ass}.    \label{fig:momcomp}}
\end{center}   
\end{table}

In general, direct calculation shows that for any $n,p,q,T$ such that $r \geq 1$, then writing $X$, $Z$, $Y$ and $H$ for the geometric, negative binomial, Poisson and binomial, respectively, defined in Table \ref{fig:momcomp}, we have 
\begin{equation} \label{eq:momcompare}
\fm{s}{X} \geq \fm{s}{Z} \geq \fm{s}{Y} \geq \fm{s}{H}
\mbox{ for all $s \geq 1$}.\end{equation}
This follows by rewriting the expressions in \eqref{eq:momcompare} in the form
$$ s! \left( \frac{r(1-q)}{q} \right)^s
\geq \frac{\Gamma(s+r)}{\Gamma(r)} \left( \frac{(1-q)}{q} \right)^s \geq
\left( \frac{r(1-q)}{q} \right)^s \geq \left( \frac{r(1-q)}{q} \right)^s \frac{\binom{n-k}{s} s!}{(n-k)^s}
$$
which simplifies to
$$ s! r^s \geq (s+r-1)(s+r-2) \ldots r \geq r^s \geq r^s \frac{ (n-k) \ldots (n-k-s+1)}{(n-k)^s}.$$
Recall from \eqref{eq:momlowbd} above that $\fm{s}{G} \geq \fm{s}{H}$ for all $s$. However, it is not the case that $\fm{s}{G} \geq \fm{s}{Y}$ for all $s$, since $\fm{s}{G} = 0$ for $s > n-k$, whereas $\fm{s}{Y} > 0$ for all $s$, although from Table \ref{fig:momcomp} it appears that $\fm{s}{G} \geq \fm{s}{Y}$ for small $s$ at least.

In the sparse regime described above (in which $k = n^\theta$, $p=1/k$ and $T = (c/q_0) k \log(n)$), we note that the dispersion parameter $r$ of \eqref{eq:paramest} tends to infinity as $n\to\infty$. This is equivalent to the fact that $M_{(2)}(G)/M_{(1)}(G)^2$ tends to $1$. We can deduce that this holds from \eqref{eq:paramest} by writing
 \begin{align}
 \frac{M_{(2)}(G)}{M_{(1)}(G)^2} & = \frac{(n-k)(n-k-1) \left(1- q_0 (2p-p^2) \right)^T}{(n-k)^2 \left(1- q_0 p \right)^{2T}} \nonumber \\
 & \simeq  \left( 1 + \frac{p^2 q_0(1-q_0)}{(1-q_0 p)^2}
 \right)^T  \nonumber \\
 & \simeq \exp \left(  \frac{p^2 T q_0(1-q_0)}{(1-q_0 p)^2} \right) \label{eq:momratio} \\
 & \simeq \exp \left( \frac{c(1-q_0) \log(n)}{k} \right) \rightarrow 1, \nonumber
 \end{align}
 since $q_0 p^2 T = c \log(n)/k$. 
 
 Similarly, in the linear regime ($k=\beta n$, $T= c n$)
 using \eqref{eq:momratio} we obtain
 $$  \frac{M_{(2)}(G)}{M_{(1)}(G)^2} \simeq
\exp \left( \frac{ c q_0(1-q_0)}{\beta^2 n} \right) \rightarrow 1, $$
since $p^2 T = c/(\beta^2 n)$.
  A Poisson approximation to the distribution of $G$ may thus be appropriate in the large-$n$ limit for the sparse and linear regimes, but the numerical results of this section make it clear that a negative binomial approximation is a more natural choice for finite blocklength applications.

\subsection{Convergence of falling moments}\label{sec:moments}

Next, we show that, in this framework, matching the first two moments of $G$ and $Z$ ensures that all the falling moments converge.

We can see this informally by simplifying \eqref{eq:fmg} and \eqref{eq:fmnb} respectively. The former gives (to leading order)
\begin{align}
\fm{s}{G} & = \binom{n-k}{s} s! \left( 1 - q_0 \left( 1- (1-p)^s \right) \right)^T \nonumber \\
& \simeq \left( n-k \right)^s \left( 1 - q_0 ( p s ) \right)^T
\nonumber \\
& \simeq \left( n-k \right)^s \exp \left( - q_0  p s T \right)
= \left( (n-k) \exp(-q_0 p T) \right)^s, \label{eq:FMG}
\end{align}
using the fact that $1-(1-p)^s = 1 - (1-p s + O(p^2)) = ps + O(p^2)$.
The latter gives
\begin{align}
\fm{s}{Z} & = \frac{\Gamma(s+r)}{\Gamma(r)}\left(\frac{1-q}{q}\right)^{s}  \simeq  \left( \frac{r(1-q)}{q} \right)^s. \label{eq:FMZ}
\end{align}
Note that the moment matching condition of \eqref{eq:matchmom1} gives that $r(1-q)/q = (n-k) (1-q_0 p)^T \simeq (n-k) \exp(-q_0 p T)$, meaning that \eqref{eq:FMG} and \eqref{eq:FMZ} agree. A more formal comparison of falling moments is given in the following theorem, where we recall that with the usual choice $p=1/k$ we have $q_0=(1-p)^k\approx e^{-1}$:

\begin{theorem} \label{thm:momconv} Consider the number of intruding defectives $G$ and negative binomial $Z$ with parameters given by moment matching, satisfying \eqref{eq:paramest}.  Under a Bernoulli test design, for any integer $s \geq 1$, if we write $C = q_0(1-q_0)/(1-q_0 p)^2$ then the moment ratio satisfies
\begin{equation} 
\frac{\fm{s}{G}}{\fm{s}{Z}} \geq
\left( \frac{ n-k-s}{(n-k)(1 + (s-1)/(2r))} \right)^s
\left( 1 + \frac{1}{2 }s(s-1) C p^2
\left( 1 - \frac{(s-2) (1-2q_0) p}{3 (1-q_0  p)} \right) \right)^T, \label{eq:bdratioL}
\end{equation} 
and 
 \begin{equation}
\frac{\fm{s}{G}}{\fm{s}{Z}} \leq
\exp \left( s(s-1) C p^2 T (1-q_0 p)^{2-s} \right).
 \label{eq:bdratioU}
\end{equation}
\end{theorem}
\begin{proof} See Appendix \ref{sec:momconv}. \end{proof}

Hence, for any fixed $s$ the ratio $\fm{s}{G}/\fm{s}{Z} \rightarrow 1$, for both
\begin{enumerate}
\item the sparse regime with $k = n^{\theta}$ for some $\theta \in (0,1)$ and $T= c e k \log n$, and
\item the linear regime with $k = \beta n$ and $T = c n$. 
\end{enumerate}
Here we control the upper bound in \eqref{eq:bdratioU}
using the fact that (see Section \ref{sec:nbapprox} above)
the $p^2 T$ is $c \log n/(q_0 k)$ or $c/(\beta^2 n)$ respectively. Similarly, we control the lower bound \eqref{eq:bdratioL} using the fact that in both regimes the $n-k$ and dispersion parameter $r$ tend to infinity (again see Section \ref{sec:nbapprox}).

Additionally, the bounds \eqref{eq:bdratioL} and \eqref{eq:bdratioU} allow us to control the ratio $\fm{s}{G}/\fm{s}{Z}$ in the finite blocklength regime, deducing bounds that can be compared with the concrete values given in Table \ref{fig:momcomp} for example. 
% Note that $1/r \simeq (1+ p^2 C)^T -1$

\subsection{Stein--Chen method} \label{sec:scmethod}

Having seen in Sections \ref{sec:nbapprox} and \ref{sec:moments} that a negative binomial distribution seems to be a reasonable approximation for the distribution of $G$, in this section we adapt the Stein--Chen method to prove explicit error bounds in the approximation of $G$ by a negative binomial distribution. We emphasise that these bounds apply for any finite blocklength application. 
The error in our approximation of $G$ by $Z$ will be measured in total variation distance, defined by
\begin{equation} \label{eq:dtv}
d_{TV}(G,Z)=\sup_{A\subseteq\ZZ^+}|\pr(G\in A)-\pr(Z\in A)|=\inf_{(G,Z)}\pr(G\not=Z)\,,
\end{equation}
where $\ZZ^+=\{0,1,\ldots\}$ and the infimum is taken over all couplings of $G$ and $Z$.

First, we briefly review the Stein--Chen method in the context of negative binomial approximation. For a more detailed introduction to this technique more generally, we refer the reader to \cite{RossSurvey}.
Recall from \cite{brown3} that $Z \sim \mbox{NB}(r,q)$ if and only if 
\begin{equation} \label{eq:steinchar}
 \ep \left[ (1-q) (r+Z) g(Z+1) - Z g(Z) \right] = 0,\end{equation}
for all test functions $g:\mathbb{Z}^+\to\mathbb{R}$ for which the expectation exists. This follows as a consequence of the fact that the negative binomial probability mass function of \eqref{eq:negbinpar}
 satisfies $z \pr(Z = z) = (1-q) (r+z-1) \pr(Z = z-1)$.

%Note that we can verify \eqref{eq:steinchar} directly for the case where $g(w) = w!/(w-k)!$ is the falling factorial function, since the value of the falling moment $\fm{k}{Z}$ given in Lemma \ref{lem:fmnb} satisfies $q \fm{k+1}{Z} = (1-q)(k+r) \fm{k}{Z}$.

The key to the analysis is that for each set $A \subseteq \ZZ^+$ we can define a function $f_A:\mathbb{Z}^+\to\mathbb{R}$ which satisfies $f_A(0)=0$ and the Stein--Chen equation
\begin{equation} \label{eq:steinchen}
(1-q)(r+z) f_A(z+1) - z f_A(z) = \II(z \in A) - \pr(Z \in A),
\end{equation}
for all $z\in\mathbb{Z}^+$. Then for any random variable $Y$ we can take the expectation of \eqref{eq:steinchen} over $Y$ to obtain
\begin{equation} \label{eq:steinchen2}
\ep \left( (1-q)(r+Y) f_A(Y+1) - Y f_A(Y) \right) = \pr(Y \in A) - \pr(Z \in A).
\end{equation}
Note that (as expected) if $Y$ were negative binomial then both RHS
and LHS of \eqref{eq:steinchen2} would be zero (the latter due to the characterization in \eqref{eq:steinchar}). However, we can deduce that if the LHS of \eqref{eq:steinchen2} is small, then so is the RHS. Indeed, if the LHS of \eqref{eq:steinchen2} is small uniformly over choices of $A$ then we can deduce a bound in total variation distance since, combining \eqref{eq:dtv} and \eqref{eq:steinchen}, we have
\begin{equation}
d_{TV}(Y, W) = \sup_{A \subseteq\ZZ^+}  \left| \ep \left( (1-q)(r+Y) f_A(Y+1) - Y f_A(Y) \right) \right|. \label{eq:steinchen3}
\end{equation}

Having reviewed the Stein--Chen method in general, we will now 
describe how we bound \eqref{eq:steinchen2} in this specific case.
For our approximating negative binomial distribution for $G$, we will make the following choices of the parameters $q$ and $r$:
\begin{equation}\label{eq:NBparams}
q=\frac{\mu}{\sigma^2}\,,\qquad\mbox{ and }\qquad r=\frac{\mu^2}{\sigma^2-\mu}=\frac{1}{e^{Tp^2 q_0}-1}\,,
\end{equation}
where 
\[
\mu=(n-k)e^{-Tp q_0}\,,\qquad\mbox{ and }\qquad\sigma^2=(n-k)^2\left(e^{-Tp(2-p)q_0}-e^{-2Tp q_0} \right)+(n-k)e^{-Tp q_0},
\]
and where, as before, we write $q_0 = (1-p)^k$.

We remark that these parameter choices \emph{do not} match the first two moments of $Z$ with those of $G$, unlike those of \eqref{eq:paramest}. Instead, the parameters $q$ and $r$ are chosen to match the first two moments of $Z$ with those of $G^{\prime\prime}$, to be defined precisely below, in which the binomial mixture which defines $G$ is replaced by a particular Poisson mixture. This may seem a little unnatural at first, but makes sense in the setting of the proof in Appendix \ref{app:scproof} below, and the ultimate effect should be negligible since the distributions of $G$ and $G^{\prime\prime}$ are close, as our proof demonstrates.

Our main result is the following
\begin{theorem}\label{thm:approx}
Let $G$ be as above, and let $Z$ have a negative binomial distribution with parameters $q$ and $r$ given by (\ref{eq:NBparams}). Then, defining $K = e^{T p q_0}$, we have
\begin{multline}\label{eq:steinthm}
d_{TV}(G,Z)\leq 
2\min\left\{\frac{q_0}{4\sqrt{1-q_0}},\frac{1}{\sqrt{T}}\alpha(q_0)+\frac{1}{\sqrt{2\pi e}}\log\left(\frac{1}{\sqrt{1-q_0}}\right)\right\}+e^{-Tpq_0}
\\
+\frac{(2-q)(n-k)}{1-q}
\Bigg(e^{r+1} K^r \exp(-K r)
+\int_0^1\left|\widehat{\Gamma}\left(\left\lceil\frac{\log(x)}{\log(1-p)}\right\rceil,Tq_0 \right)-\widehat{\Gamma}\left(r,K r x\right)\right|\,dx\Bigg)\,,
\end{multline}
where
\begin{align*}
\alpha(q_0)&=\frac{0.4748\left[\sqrt{1-q_0}\left(1+2q _0^2 e^{-q_0}\right)+ q_0^2+(1-q_0)^2\right]}{\sqrt{q_0(1-q_0)}}\,,
\end{align*}
and $\widehat{\Gamma}(\cdot,\cdot)$ is the normalised upper incomplete gamma function, defined by 
\[
\widehat{\Gamma}(s,y)= \frac{1}{\Gamma(s)} \int_y^\infty t^{-s-1}e^{-t}\,dt\,,
\] 
where $\Gamma(\cdot)$ is the gamma function.
\end{theorem} 
\begin{proof} See Appendix \ref{app:scproof} below. \end{proof}

We conclude this section with some discussion and numerical illustration of the bound of Theorem \ref{thm:approx}. We will discuss further aspects of the convergence of this bound in Remark \ref{rem:gammbd} below, but first use numerical illustrations to gain some initial understanding of its behaviour. Firstly, we note that although our bound applies for any finite blocklength, there are examples in which it is worse than the trivial bound $d_{TV}(G,Z)\leq1$, or performs poorly compared to the simulation results we observed in Section \ref{sec:nbapprox}. For example, with $n=500$, $k=10$, $p=0.1$ and $T=100$, the bound of Theorem \ref{thm:approx} gives $d_{TV}(G,Z)\leq1.80$, which is clearly uninformative. We give some further examples of our upper bound in Table \ref{tab:stein}, all in the case $n=2500$ and for various values of $k$, $p$ and $T$, from which it is clear that there are some cases in which the bound of Theorem \ref{thm:approx} performs very well, and demonstrates proximity of the distribution of $G$ to negative binomial.

\begin{table}
    \centering
    \begin{tabular}{|l|c|c|c|c|c|c|}
        \hline
         & \multicolumn{3}{c|}{$T=500$} & \multicolumn{3}{c|}{$T=1000$} \\
         \cline{2-7}
         & $p=0.05$ & $p=0.1$ & $p=0.2$ & $p=0.05$ & $p=0.1$ & $p=0.2$ \\
         \hline
         $k=5$ & 0.501 & 0.337 & 0.120 & 0.460 & -- & -- \\
         \hline
         $k=10$ & 0.342 & 0.216 & 0.066 & 0.307 & 0.198 & 0.057 \\
         \hline
         $k=20$ & 0.234 & -- & -- & 0.200 & 0.065 & --\\
         \hline
    \end{tabular}
    \caption{The upper bound of Theorem \ref{thm:approx} in the case $n=2500$ for various values of $k$, $p$ and $T$. The symbol ``--'' indicates that the upper bound is larger than 1, and therefore uninformative.}
    \label{tab:stein}
\end{table}

It is interesting to note that in many cases the largest share of the error estimate in Theorem \ref{thm:approx} comes from the first term of the upper bound, which arises from the approximation of the binomial random variable $M_0$ by a Poisson random variable of the same mean (see the proof in Appendix \ref{app:scproof} for details). Nevertheless, the bound we would obtain without making this approximation generally performs worse than our Theorem \ref{thm:approx}. We conjecture that this is because the integral in the final term of the upper bound is made smaller by this approximation of $M_0$ (compared to the corresponding term without this approximation), resulting in a smaller upper bound overall because of the relatively large factors multiplying this integral.

\begin{remark} \label{rem:gammbd}
While the incomplete gamma functions in \eqref{eq:steinthm} make the integral in that bound not straightforward to interpret directly, we can provide an upper bound on this quantity using concentration of measure inequalities.

We split the region of integration in three parts, with breaks at $(1 \pm \epsilon)/K$. On the region $\left( (1-\epsilon)/K, (1+\epsilon)/K \right)$, we simply bound the integrand by 1, to give $2 \epsilon/K$. Using the fact that for $0 \leq u,v \leq 1$ we can bound $|u-v| \leq \max(u,v)$ and $|u-v| \leq \max(1-u,1-v)$, we can hence bound the integral in \eqref{eq:steinthm} by
\begin{multline*}
\frac{2 \epsilon}{K} + \frac{1}{K} \max \left( \pr \left( \xi^\prime < \frac{(1-\epsilon)}{K} \right), \pr \left(\eta^\prime < \frac{(1-\epsilon)}{K} \right) \right)\\
+ \max \left( \pr \left( \xi^\prime > \frac{(1+\epsilon)}{K} \right), \pr \left(\eta^\prime > \frac{(1+\epsilon)}{K} \right) \right)\,. 
\end{multline*}

Since $\eta' \sim \Gamma(r,rK)$, the inequalities \eqref{eq:chernoffbd} and \eqref{eq:chernoffbd2} below give us that we may upper bound both $\pr(\eta'>z)$ and $\pr(\eta'<z)$ by 
$$ (K z e)^r \exp(-r K z) = \exp
\left( r (1 - K z + \log(K z)) \right).$$
Hence, for example,  we know by writing
$m(s) = s - \log(1+s) \simeq s^2/2$ that
$$ \pr \left(\eta^\prime > \frac{(1+\epsilon)}{K} \right) \leq
\exp \left( - r m(\epsilon) \right) \mbox{\;\; and \;\;} \pr \left(\eta^\prime < \frac{(1-\epsilon)}{K} \right) \leq
\exp \left( -r m(-\epsilon) \right)\,.$$

A similar standard Chernoff bounding argument, as used for \eqref{eq:chernoffbd} and \eqref{eq:chernoffbd2}, gives that, for $Y \sim \po(\lambda)$, both $\pr( Y > y)$ and $\pr(Y < y)$ are bounded above by $\exp \left( - \lambda h(y/\lambda - 1) \right)$, where $h(s) = (1+s) \log(1+s) - s
\simeq s^2/2$. Hence, for example, we can bound
\begin{align*} 
\pr( \xi^\prime >x) &= \pr \left( M' < \frac{-\log x}{ -\log(1-p)} \right) \leq 
 \pr \left( M' < \frac{-\log x}{ p} \right)\\
&\leq \exp \left(- T q_0 h \left(  \frac{-\log x}{T q_0 p} - 1 \right) \right)\,,
\end{align*} 
if $-\log x/p \leq \ep M' = T q_0$, since $M' \sim \po(T q_0)$. Hence, taking $x= (1+\epsilon)/K = (1+\epsilon) e^{-T p q_0}$ we obtain
$$ \pr \left( \xi^\prime > \frac{(1+\epsilon)}{K} \right)
\leq  \exp \left(- T q_0 h \left( - \frac{\log (1+\epsilon)}{T q_0 p} \right) \right)\,.$$ 
Similarly, we can bound
$$ \pr( \xi^\prime < x) = \pr \left( M' > \frac{-\log x}{ -\log(1-p)} \right) \leq \exp \left(- T q_0 h \left(  \frac{-\log x}{-T q_0 \log(1-p)} - 1 \right) \right)\,,
$$ 
if $-\log x/-\log(1-p) \leq \ep M' = T q_0$, since $M' \sim \po(T q_0)$. Hence, taking $x= (1-\epsilon)/K = (1-\epsilon) e^{-T p q_0}$ and writing $m(-p) = - p - \log(1-p) \geq 0$ as above, we obtain that
$$ \pr \left( \xi^\prime < \frac{(1-\epsilon)}{K} \right) \leq  \exp \left( - T q_0 h \left( - \frac{-\log(1-\epsilon) - T q_0 m(-p)}{-T q_0 \log(1-p)} \right) \right)\,,$$
assuming that the numerator is positive, which holds, for example, if $\epsilon > T q_0 m(-p)$. This condition is satisfied, for example, for large $n$ in the linear regime where $T=c n$, $k = \beta n$ and $p = 1/k$, since $m(p) \simeq p^2/2$ so $T q_0 m(-p)
\sim {\rm const.}/n$ in this regime.
\end{remark}

\section{Implications for group testing algorithms} \label{sec:implications}

As discussed previously, understanding the distribution of $G$ allows us to control the performance of the conservative two-stage algorithm of Aldridge \cite{aldridge6}. Specifically,  we consider the scenario where $T_1$ tests are performed in the first stage according to a Bernoulli testing design with parameter $1/k$, and then $T_2$ individual tests are performed afterwards to resolve the status of items we are unsure about. 

Given a fixed total budget of $T$ tests, we can regard the standard Bernoulli test design as corresponding to $T_1 = T$ and $T_2 = 0$ (no second stage), and individual testing as corresponding to $T_1 =0$ and $T_2 = T$ (no first stage). However, it is natural to consider strategies intermediate to these, to see if better performance can be obtained by choosing some $T$ satisfying $0 < T_1 \leq T$.

Given an overall budget of $T$ tests, this leaves us $T_2 = T-T_1$ individual tests to find the status of
$k+G$ items, and we will succeed if $T_2 \geq k+G$ or, equivalently, if $G  + T_1 + k \leq T$. Equivalently, the algorithm will fail if $G > T - T_1-k$.

Aldridge \cite[Theorem 1]{aldridge6} considers this failure event in terms of expected values. That is, we may wish to choose $T_1$ to minimise 
\begin{align}
 \ep(T_1 + k + G) & =  T_1 + k + \fm{1}{G} =
T_1 + k + (n-k) (1-q_0 p)^{T_1} \nonumber \\
& \simeq T_1 + k + (n-k) \exp \left( - \frac{1}{e k} T_1 \right),
\end{align}
and as in \cite{aldridge6} direct calculation gives that the optimal choice of $T_1$ to control this expectation is 
\begin{equation} \label{eq:topt}
 T_1^* = k e \log \left( \frac{n-k}{ke} \right).
 \end{equation}
Interestingly, since $p=1/k$ and $q_0 \simeq 1/e$, this choice makes the expected value 
$$ \fm{1}{G} = (n-k)(1-q_0 p)^{T_1} 
\simeq (n-k) \exp \left( - \frac{T_1}{e k} \right) = k,$$
meaning that the average number of intruding non-defectives approximately equals the number of true defectives, so a randomly chosen individual test is positive with probability close to $1/2$, maximising the information that we gain from it.

However, instead of simply finding the expected value we can use the moment values calculated in this paper to bound the error probability under this kind of two-stage strategy. For example,  Chebyshev's inequality gives an upper bound on the error probability:
\begin{lemma} If we use $T_1 = k c_1$ tests in the first stage and $T_2 = k c_2$ tests in the second stage
\begin{equation} \label{eq:chebya}
\pr(\err) \leq \min \left( 1, \frac{ \fm{2}{G}/\fm{1}{G}^2 - 1  + 1/\fm{1}{G}}{ \left(n \beta(c_2 - 1)/\fm{1}{G} - 1 \right)^2}
\right). \end{equation} 
\end{lemma}
\begin{proof} A standard argument gives
\begin{eqnarray}
\pr(\err) & = & \pr( G  > T_2 - k) = \pr( G - \ep G> T_2 -k - \ep G) \nonumber \\
& \leq & \frac{\var(G)}{(T_2 - k - \ep G)^2}
= \frac{ \fm{2}{G}  - \fm{1}{G}^2 + \fm{1}{G}}{ \left(n \beta(c_2 - 1) - \fm{1}{G} \right)^2}, \nonumber
%& = & \frac{ \fm{2}{G}/\fm{1}{G}^2 - 1  + 1/\fm{1}{G}}{ \left(n \beta(c_2 - 1)/\fm{1}{G} - 1 \right)^2} \nonumber 
\end{eqnarray}
and the result follows.
\end{proof}

Note that the expression \eqref{eq:chebya} becomes
\begin{equation}
 \simeq  \frac{ \exp(c_1 q_0 p) - 1 + \exp(c_1 q_0)/n}{ \left( \beta(c_2-1) \exp(c_1 q_0)/(1-\beta) - 1 \right)^2} \label{eq:cheby}
\end{equation}
in the linear regime ($k=\beta n$), 
since $\fm{1}{G} = (n-k) (1-q_0 p)^{T_1} \simeq n(1-\beta) \exp(-c_1 q_0)$ and $\fm{2}{G} = (n-k)(n-k-1)(1-q_0 (2p-p^2))^T \simeq n^2(1-\beta)^2 \exp(-c_1 q_0 (2-p))$.

Hence if we take $c_1  = e \log( (1-\beta)/(\beta e))$ as suggested by \eqref{eq:topt}, so that $\exp(c_1 q_0) \simeq (1-\beta)/\beta$, then \eqref{eq:cheby} becomes
$$
\pr(\err) \leq \frac{ \exp(2 c_1 q_0 p) - 1 + (1-\beta)/(n \beta)}{ (c_2-2)^2},
$$
so for any fixed $c_2 > 2$ the error probability tends to zero at rate $1/n$ as $n \rightarrow \infty$.

Further, using the negative binomial approximation of this paper we can find large deviations bounds on the success probability of the two-stage algorithm. Figure \ref{figure3} shows that in this case this negative binomial approximation provides accurate bounds on the total number of tests needed.

    \begin{figure}[!ht] 
    \centering
    % Note: Can change the figure name to _3 instead of _2 if we decide to use T instead of T1 to represent the number of tests
    \includegraphics[width=0.7\textwidth]{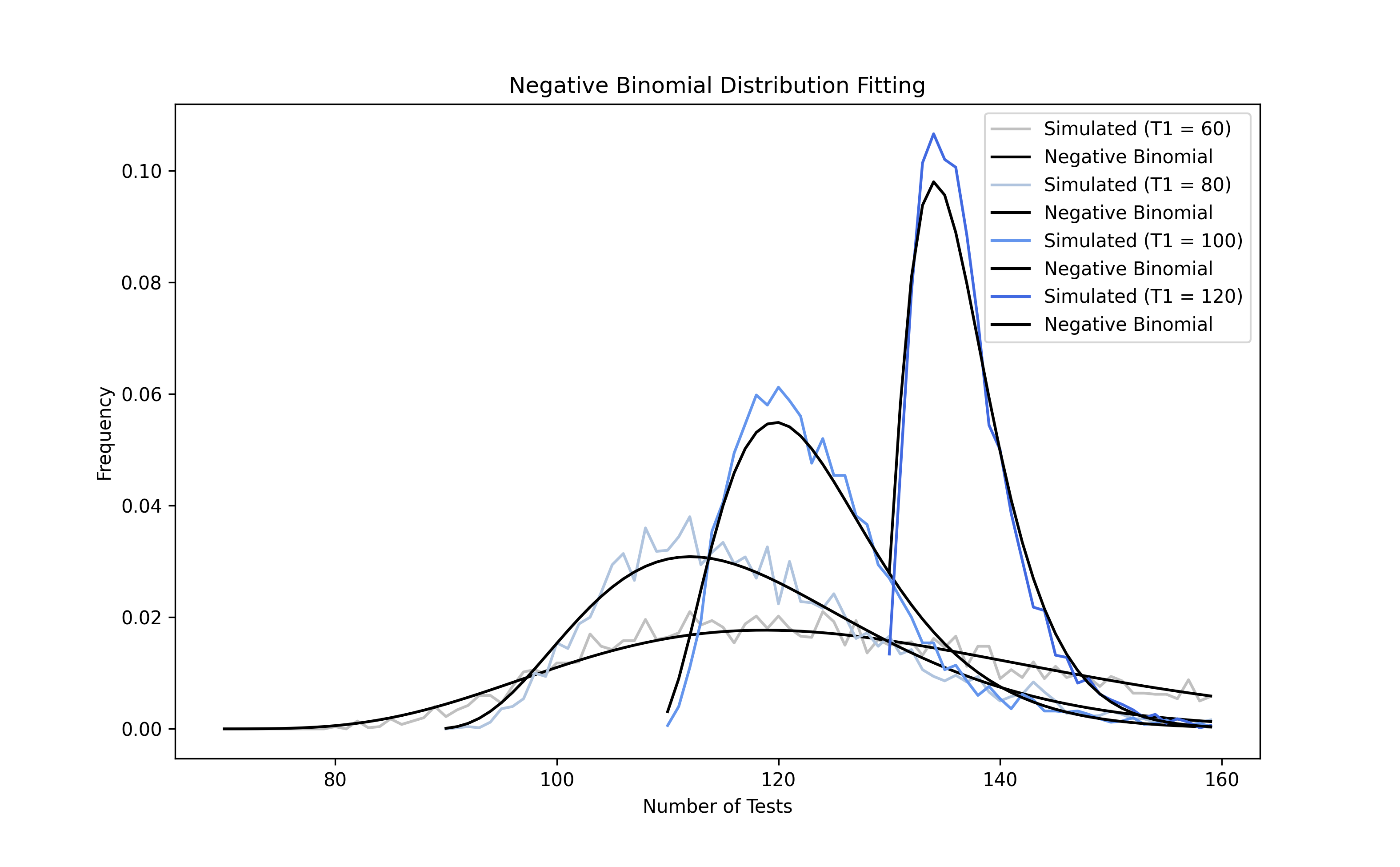} 
    \caption{Comparison of probability mass function of $G$ (via simulation) and approximating negative binomial distribution $Z$ (via calculation) for 2-Stage COMP with a variety of values of $T_1$ ($n=500$, $k=10$, $p=0.1$).} 
    \label{figure3} 
    \end{figure}

That is, if we write $Z$ for the negative binomial approximation with parameters given by \eqref{eq:paramest}, we know that for any $g$ the tail probability $\pr( G > g) \simeq \pr(Z > g)$. Using a standard large deviations argument, we know:
\begin{proposition} \label{prop:exptails} 
If $Z$ is negative binomial with parameters $r$ and $q$, then for any $g > \ep Z$: 
\begin{equation} \label{eq:LDub}
\pr(Z \geq g)  \leq \exp \left( - (g+r) D_{KL} \left( \left. \frac{g}{g+r} \right\| 1-q \right) \right), 
\end{equation}
where we write $D_{KL}(v \| w) =  v \log_e(v/w) + (1-v) \log_e((1-v)/(1-w))$ for the Kullback--Leibler divergence from a Bernoulli($v$) random variable to a Bernoulli($w$).
\end{proposition}
\begin{proof} See Appendix \ref{app:proofexptails}. \end{proof}

Hence, again in the linear scenario ($k=\beta n$), taking $T_1 = k \beta e \log( (1-\beta)/\beta e)$ tests in the first stage (as suggested by \eqref{eq:topt}) and $T_2 = k c_2$ in the second,  then the probability of failure will be 
\begin{equation}
 \pr( G > T_2-k) = \pr \left( G > k (c_2-1) \right) \simeq \pr \left( Z > k (c_2 -1) \right),
 \end{equation}
which we can bound using Proposition \ref{prop:exptails}. Using the explicit bounds above, we obtain an upper bound decaying exponentially in $k$.

\begin{figure}[!ht] 
    \centering
    % Note: Can change the figure name to _3 instead of _2 if we decide to use T instead of T1 to represent the number of tests
    \includegraphics[width=0.9\textwidth]{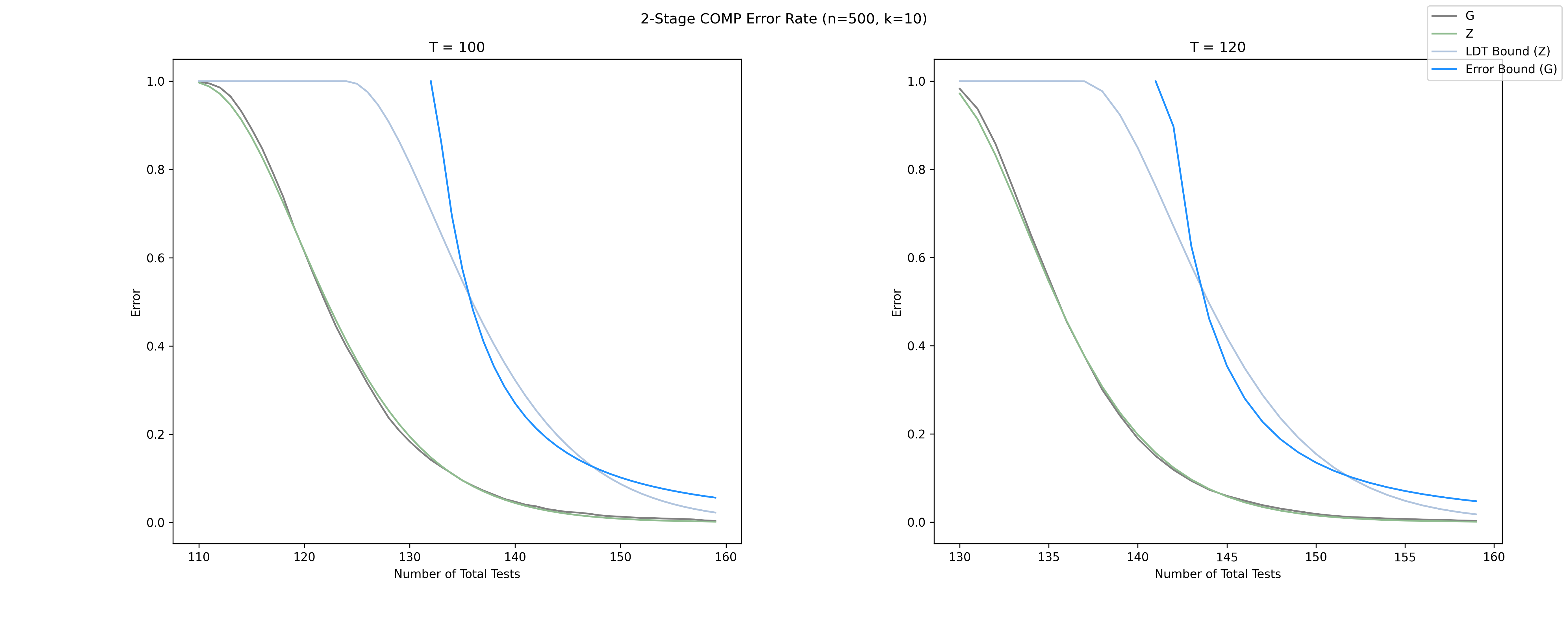} 
    \caption{Comparison of error probability for 2-Stage COMP - via simulation for $G$ and direct calculation of negative binomial approximation $Z$. Upper bounds implied by \eqref{eq:chebya} and \eqref{eq:LDub} are provided for comparion ($n=500$, $k=10$, $p=0.1$).} 
    \label{figure4} 
    \end{figure}
    
In Figure \ref{figure4} we see that the negative binomial approximation $Z$ well approximates the distribution of $G$ and that the upper bounds implied by \eqref{eq:chebya} and \eqref{eq:LDub} are somewhat tight.

This analysis could be extended to cover a two-stage version of the DD algorithm of \cite{johnson33}. In Stage 1 of this algorithm there would potentially be some items which could be confirmed as `definitely defective' (because they appear in some positive test only otherwise containing items which are guaranteed by other tests to be non-defective). Such items would not need individual testing in Stage 2, which could potentially reduce the number of tests required by up to $k$; this could be significant in the linear regime at least. However, we leave this question as further work, for reasons of space and due to the complexity of the analysis of non-adaptive DD in \cite{johnson33}.

\section{Conclusion}\label{sec:conc}

We have identified the key role played by $G$, the number of intruding non-defective items, in group testing algorithms. Under the standard Bernoulli testing strategy, we have identified the distribution of $G$, given explicit expressions for its falling moments and shown how it can be well approximated by a negative binomial distribution with given parameters. This allows us to deduce results concerning the performance of the COMP and DD group testing algorithms.

\appendix

\section{Proof of Proposition \ref{prop:ass}} \label{app:proofass}

\begin{proof}
We use Proposition 1 of \cite{fortuin}, which shows that it is enough to verify that the $G_i$ satisfy the so-called FKG condition:
\begin{equation} \label{eq:fkg}
\pr( \vc{G} = \vc{x} \vee \vc{y}) \pr( \vc{G} = \vc{x} \wedge \vc{y})
\geq \pr( \vc{G} = \vc{x}) \pr( \vc{G} = \vc{y}),\end{equation}
where  $\vee$ and $\wedge$ represent the maximum and minimum 
respectively. 

By independence, we can describe the distribution of $\vc{G}$ conditional on $M_0$, the number of negative tests. Recall that $M_0$ has a binomial distribution with parameters $T$ and $q_0$. For any $\vc{g}$, we can write
\begin{align*} \pr( \vc{G} = \vc{g}) & = \sum_m \pr(M_0=m) P_m^{w(\vc{g})} (1-P_m)^{n-k-w(\vc{g})} \\
& =  \sum_m \pr(M_0=m) (1-P_m)^{n-k} R_m^{w(\vc{g})} ,
\end{align*}
where $P_m = (1-p)^m$ and $R_m = P_m/(1-P_m)$. This follows since we can check which tests the defective items appear in first, and then each non-defective item is independently intruding with probability $P_m$, since it must not appear in the $m$ negative tests.

Using this expression, and writing $\pr( \vc{G} = \vc{x} \vee \vc{y}) \pr( \vc{G} = \vc{x} \wedge \vc{y})$ in the form
$$ \left( \sum_m \pr(M_0 = m) \pr(  \vc{G} = \vc{x} \vee \vc{y} | M_0 =m) \right) \left( \sum_\ell \pr(M_0 = \ell) \pr(  \vc{G} = \vc{x} \wedge \vc{y} | M_0 =\ell) \right)$$
we can verify the FKG condition \eqref{eq:fkg} by writing
\begin{align*}
\lefteqn{  \pr( \vc{G} = \vc{x} \vee \vc{y}) \pr( \vc{G} = \vc{x} \wedge \vc{y}) - \pr( \vc{G} = \vc{x}) \pr( \vc{G} = \vc{y}) 
 } \\
%& = &  \left\{ \left( \sum_{m} \pr(M_0 = m) R_m^{w(\vc{x} \vee \vc{y})} \right) \left( \sum_{\ell} \pr(M_0 = \ell) R_\ell^{w(\vc{x} \wedge \vc{y})} \right) \right. \\
%& & \left. - \left( \sum_{m} \pr(M_0 = m) R_m^{w(\vc{x})} \right)  \left( \sum_{m} \pr(M_0 = \ell) R_\ell^{w(\vc{y})} \right) \right\} \\
& = \sum_{m,\ell} 
\pr(M_0 =m) \pr(M_0= \ell) (1-P_m)^{n-k} (1-P_\ell)^{n-k} \left\{ 
R_m^{w(\vc{x} \vee \vc{y})} R_\ell^{w(\vc{x} \wedge \vc{y})}
- R_m^{w(\vc{x})} R_\ell^{w(\vc{y})}
\right\} \\
& =:  \sum_{m,\ell} 
\pr(M_0 =m) \pr(M_0= \ell) (1-P_m)^{n-k} (1-P_\ell)^{n-k}
\alpha_{m,\ell}(\vc{x}, \vc{y}). 
\end{align*}
We can pair up these $\alpha_{m,\ell}$ terms, noticing the fact 
that $w(\vc{x} \vee \vc{y}) + w(\vc{x} \wedge \vc{y}) = w(\vc{x}) +  w(\vc{y})$ means that $\alpha_{m,m}$ vanishes, and that in general
we can write
$$ \alpha_{m,\ell} + \alpha_{\ell,m} = \frac{1}{R_m^{w_+} R_\ell^{w_+}}
\left( R_m^{w_+} R_\ell^{w(\vc{x})} - R_\ell^{w_+} R_m^{w(\vc{x})}
\right)\left( R_m^{w_+} R_\ell^{w(\vc{y})} - R_\ell^{w_+} R_m^{w(\vc{y})}
\right),
$$
where for brevity we write $w_+ = w(\vc{x} \vee \vc{y})$ and $w(\vc{x} \wedge \vc{y}) 
= w(\vc{x}) + w(\vc{y}) - w_+$. Observe that since $w_+ \geq w(\vc{x})$ and $w_+ \geq w(\vc{y})$ both these bracketed terms have the same sign, so $\alpha_{m,\ell} + \alpha_{\ell,m} \geq 0$ and the FKG condition is satisfied.
\end{proof}

\section{Proof of Theorem \ref{thm:momconv}} \label{sec:momconv}

\begin{proof} Write $L = (n-k)$ for the number of non-defective items.
We know that
\begin{align*}
\fm{s}{G} & = \frac{L!}{(L-s)!} \left( 1 - q_0 \left(1-(1-p)^s
\right) \right)^T, \\
\fm{s}{Z} & = \frac{\Gamma(r+s)}{\Gamma(r)} \left( \frac{1-q}{q}
\right)^s,
\end{align*}
where we write $q_0 = (1-p)^k$ for brevity. By moment matching, we know that $L(1-q_0 p)^T = r(1-q)/q$, so the ratio
\begin{align}
\frac{ \fm{s}{G}}{\fm{s}{Z}} & = \left( \frac{L!}{(L-s)! L^s}
\frac{\Gamma(r)r^s}{\Gamma(r+s)} \right) \left(
\frac{ 1 - q_0 \left(1-(1-p)^s \right)}{(1-q_0 p)^s} \right)^T\nonumber \\
& = \left( \frac{L!}{(L-s)! L^s}
\frac{\Gamma(r)r^s}{\Gamma(r+s)} \right) \left( (1-q_0)
\left( \frac{1}{1-q_0 p} \right)^s + q_0 \left( \frac{1-p}{1-q_0 p}
\right)^s
\right)^T. \label{eq:tohandle}
\end{align}
We will treat the two bracketed (falling factorial and $T$th power) terms of \eqref{eq:tohandle} separately.

\begin{enumerate}
    \item {\bf Falling factorial term} 
Note that
\begin{equation} \label{eq:fallfact}
\frac{L!}{(L-s)! L^s}
\frac{\Gamma(r)r^s}{\Gamma(r+s)} = \frac{L(L-1) \ldots (L-s+1)}{L^s}
\frac{r^s}{r(r+1) \ldots (r+s-1)} \leq 1,
\end{equation}
by a termwise comparison. Similarly, we can bound \eqref{eq:fallfact} from below using the arithmetic mean-geometric mean inequality as
\begin{eqnarray*}
 \frac{L!}{(L-s)! L^s}
\frac{\Gamma(r)r^s}{\Gamma(r+s)} & \geq & \left( \frac{L-s}{L} \right)^s \frac{1}{\prod_{i=0}^{s-1} (1+i/r)} \\
& \geq & \left( \frac{L-s}{L} \right)^s \frac{1}{ \left( \frac{1}{s} \sum_{i=0}^{s-1} (1+i/r) \right)^s}
= \left( \frac{ L-s}{L(1 + (s-1)/(2r))} \right)^s.
\end{eqnarray*}

\item {\bf $T$th power term}
We can write the second term of \eqref{eq:tohandle} as 
$(1-R)^T$, where we write $\ol{q} = 1-q_0 p$,
\begin{equation}
 R =  (1-q_0) \left(1 -  \left(\frac{1}{\ol{q}} \right)^s \right) + q_0 \left( 1-  \left( \frac{1-p}{\ol{q}} \right)^s \right). \label{eq:rdef}
 \end{equation}
We first provide an upper bound on $R$ using the fact that
$\theta(x) := (1+x)^s - \left( 1 + x s + x^2 s(s-1)/2 + x^3 s(s-1)(s-2)/6 \right) \geq 0$
(this result follows using the fact that $\theta(0) = \theta'(0)$ and since $\theta''(x) = s(s-1) \left( (1+x)^{s-2} - x(s-2) -1 \right) \geq 0$ by Bernoulli's inequality). Equivalently, for any $x$  we can write $1-(1+x)^s \leq - s x -  \frac{1}{2} s(s-1) \left(x^2 + x^3 (s-2)/3 \right) $. 

Hence, taking $x_1 = 1/\ol{q} - 1 = p q_0/\ol{q}$ and $x_2=(1-p)/\ol{q}-1 = -p(1-q_0)/\ol{q}$ respectively in the two terms of  \eqref{eq:rdef} this gives
$$
R \leq - \frac{1}{2 }s(s-1) C p^2
\left( 1 - \frac{(s-2) (1-2q_0) p}{3 \ol{q}} \right),
$$
since the linear terms cancel as $(1-q_0) x_1 + q_0 x_2 = 0$, and
where we recall that we write $C = q_0(1-q_0)/\ol{q}^2$.

We can give a complementary lower bound on $R$ using a standard argument based on the mean value theorem with $f(t) = t^s$ by rewriting \eqref{eq:rdef} to obtain
\begin{align}
 R 
& = (1-q_0) \left( f(1) -  f \left( 1 + \frac{ p q_0}{\ol{q}} \right) \right) + q_0 \left( f(1) - f \left( 1 - \frac{ p (1-q_0)}{\ol{q}} \right) \right) \nonumber \\
& = -(1-q_0) \frac{p q_0}{\ol{q}} f'(\beta) +
q_0 \frac{p(1-q_0)}{\ol{q}} f'(\alpha) \nonumber \\
& = -\frac{p q_0(1-q_0)}{\ol{q}} (\beta - \alpha) f''(\gamma),
\label{eq:rbound}
\end{align}
for some $\alpha \in (1 - \frac{ p (1-q_0)}{\ol{q}},1)$, $\beta \in (1, 1 + \frac{ p q_0}{\ol{q}})$ and  $\gamma \in (\alpha,\beta)$. Then since we know $(\beta-\alpha) \leq \frac{p}{\ol{q}}$ and $\gamma \leq \beta \leq 1 + p q_0/\ol{q}$, the expression \eqref{eq:rbound} gives
$$ R \geq -\frac{p^2 q_0(1-q_0)}{\ol{q}^2} s(s-1) \left(1+ \frac{p q_0}{\ol{q}} \right)^{s-2} = - s(s-1) C p^2 \left( \frac{1}{\ol{q}} \right)^{s-2},$$ 
meaning that $$(1-R)^T \leq 
\exp \left( s(s-1) C p^2 T q_0^{2-s} \right),$$
and the proof is complete.
\end{enumerate}
\end{proof}

\section{Proof of Theorem \ref{thm:approx}} \label{app:scproof}

In proving the theorem, our strategy will be to replace $G$ by the mixed Poisson version $G^{\prime\prime}$ defined below (bounding the error in making this replacement). We then approximate $G^{\prime\prime}$ by a negative binomial distribution by noting that a negative binomial can itself be written as a mixed Poisson with gamma mixing distribution. Lemma \ref{lem:NBtoG} below allows us to transfer our negative binomial approximation problem for a mixed Poisson distribution into a gamma approximation problem for the mixing distribution. We may then bound the appropriate distance from our mixing distribution to gamma to complete the proof.

Before proceeding with this programme, we first define a further metric we will need: the Wasserstein distance, denoted by $d_W$. For any non-negative, real-valued random variables $X$ and $Y$, we define
\begin{equation} \label{eq:defwass}
d_W(X,Y)=\sup_{h\in\mathcal{H}_W}|\ep h(X)-\ep h(Y)|=\int_0^\infty|\pr(X\leq x)-\pr(Y\leq x)|\,dx\,,
\end{equation}
where $\mathcal{H}_W$ is the set of absolutely continuous functions $h:\mathbb{R}^+\to\mathbb{R}$ with $\|h^\prime\|\leq1$, and $\|\cdot\|$ is the supremum norm defined by $\|g\|=\sup_x|g(x)|$ for any real-valued function $g$.

\begin{lemma}\label{lem:NBtoG}
Let $Z$ have a negative binomial distribution with parameters $r$ and $q$, and let $H$ have a mixed Poisson distribution, $H|\xi\sim\po(\xi)$ for some positive random variable $\xi$.
Let $\eta\sim\Gamma(r,\lambda)$ have a gamma distribution with density function $\frac{\lambda^r}{\Gamma(r)}x^{r-1}e^{-\lambda x}$, for $x>0$, where $\lambda=\frac{q}{1-q}$. 
Then
\[
d_{TV}(H,Z)\leq\frac{2-q}{1-q} d_W(\xi,\eta)\,.
\]
\end{lemma}  
\begin{proof}
It can be easily checked by direct calculation that $Z$ has the mixed Poisson distribution $Z|\eta\sim \po(\eta)$. Following Stein's method for  negative binomial approximation (see \cite{barbour7,brown3,ross2013} and the review in Section \ref{sec:scmethod}), we let $f=f_A$ satisfy $f(0)=0$ and (see
\eqref{eq:steinchen})
\[
(1-q)(r+j)f(j+1)-jf(j) = I(j\in A)-\pr(Z\in A)\,,
\] 
where $A\subseteq\ZZ^+$, so that we may write (see \eqref{eq:steinchen3})
\[
d_{TV}(H,Z)=\sup_{A\subseteq\ZZ^+}|\ep[(1-q)(r+H)f(H+1)-Hf(H)]|\,.
\]
We note the following bounds on $f$, taken from Lemma 3 of \cite{brown3} and Lemmas 2.2 and 2.3 of \cite{ross2013}, respectively:
\begin{equation}\label{eq:NBmagic}
\sup_j|f(j)|\leq\frac{1}{1-q}\,,\qquad|\Delta f(j)|\leq\frac{1}{j}\,,\qquad\sup_j|\Delta(D^{(r)}f)(j)|\leq\frac{2-q}{(1-q)r}\,,
\end{equation}
where $\Delta f(j)=f(j+1)-f(j)$ and $D^{(r)}f(j)=\left(\frac{j}{r}+1\right)f(j+1)-\frac{j}{r}f(j)$, so that
\[
\Delta(D^{(r)}f)(j)=\left(\frac{j+1}{r}+1\right)\Delta f(j+1)-\frac{j}{r}\Delta f(j)\,.
\]
Now, we define $g(x)=(1-q)\ep[f(H+1)|\xi=x]$. Using the fact that $H$ has a mixed Poisson distribution, a direct calculation shows that $g^\prime(x)=(1-q)\ep[\Delta f(H+1)|\xi=x]$, and similarly for the second derivative of $g$.
We also note that (see page 12 of \cite{barbour} for example), since $H$ has a mixed Poisson distribution, 
\begin{equation}\label{eq:MP}
\xi\ep[f(H+1)|\xi]=\ep[Hf(H)|\xi]\,.
\end{equation} This then allows us to write
\begin{align*}
\ep[(1-q)&(r+H)f(H+1)-Hf(H)]\\
&=\ep\ep[(1-q)(r+H)f(H+1)-Hf(H)|\xi]\\
&=\ep\ep[(1-q)rf(H+1)+(1-q)\xi f(H+2)-\xi f(H+1)]\\
&=\ep[\xi g^\prime(\xi)+(r-\lambda\xi)g(\xi)]\,.
\end{align*}
This latter expression is closely related to Stein's method for gamma approximation, as developed by Luk \cite{luk1994}; see also \cite{gaunt2017} and references therein for more recent developments. In particular, it is known that since $\eta$ has a gamma distribution, $\ep[\eta g^\prime(\eta)+(r-\lambda\eta)g(\eta)]=0$. Letting $h(x)=xg^\prime(x)+(r-\lambda x)g(x)$ (and noting that our earlier calculations show that $h$ is differentiable), we may therefore write
\[
d_{TV}(H,Z)=\sup_{A\subseteq\ZZ^+}|\ep h(\xi)-\ep h(\eta)|\leq\sup_{A\subseteq\ZZ^+}\|h^\prime\|d_W(\xi,\eta)\,.
\]
To complete the proof, it remains only to bound $|h^\prime(x)|\leq\frac{2-q}{1-q}$. To that end, we note that
\begin{align*}
h^\prime(x)&=xg^{\prime\prime}(x)+g^\prime(x)+(r-\lambda x)g^\prime(x)-\lambda g(x)\\
&=(1-q)\big(x\ep[\Delta^2f(H+1)|\xi=x]+(1+r-\lambda x)\ep[\Delta f(H+1)|\xi=x]\\
&\hspace{20pt}-\lambda\ep[f(H+1)|\xi=x]\big)\\
&=(1-q)\big(\ep[H\Delta^2f(H)|\xi=x]+(1+r)\ep[\Delta f(H+1)|\xi=x]-\lambda\ep[H\Delta f(H)|\xi=x]\\
&\hspace{20pt}-\lambda\ep[f(H+1)|\xi=x]\big)\\
&=(1-q) r\ep[\Delta(D^{(r)}f)(H)|\xi=x]\\
&\hspace{20pt}- (1-q) \lambda\left(\ep[H\Delta f(H)|\xi=x]+\ep[f(H+1)|\xi=x]\right)\,,
\end{align*}
where the penultimate inequality again uses (\ref{eq:MP}). Using the bounds (\ref{eq:NBmagic}), we therefore have
\[
|h^\prime(x)|\leq\frac{(1-q)r(2-q)}{(1-q)r}+(1-q)\lambda\left(1+\frac{1}{1-q}\right)=2-q+(1-q)\lambda+\lambda=\frac{2-q}{1-q}\,,
\]
as required, since $\lambda = q/(1-q)$.
\end{proof}

\begin{proof}[Proof of Theorem \ref{thm:approx}]
We now use Lemma \ref{lem:NBtoG} to establish Theorem \ref{thm:approx}. Recalling that $M_0\sim\bin(T,q_0)$, we define $M^\prime\sim\po(Tq_0)$. Similarly, recalling that $G$ has the mixed binomial distribution $G|M_0\sim \bin(n-k,(1-p)^{M_0})$, we define $G^\prime$ and $G^{\prime\prime}$ as follows:
\begin{align*}
G^\prime|M^\prime&\sim \bin(n-k,(1-p)^{M^\prime})\,,\\
G^{\prime\prime}|M^\prime&\sim\po((n-k)(1-p)^{M^\prime})\,.
\end{align*}
We then write
\[
d_{TV}(G,Z)\leq d_{TV}(G,G^\prime)+d_{TV}(G^\prime,G^{\prime\prime})+d_{TV}(G^{\prime\prime},Z)\,,
\]
and bound each of these three terms separately.

Firstly, we note that
\[
d_{TV}(G,G^\prime)\leq d_{TV}(M_0,M^\prime)\leq2\min\left\{\frac{q_0}{4\sqrt{1-q_0}},\frac{1}{\sqrt{T}}\alpha(q_0)+\frac{1}{\sqrt{2\pi e}}\log\left(\frac{1}{\sqrt{1-q_0}}\right)\right\}\,,
\]
where the final inequality comes from the main result of Weba \cite{weba1999} combined with the sharpened value 0.4748 of the constant in the Berry--Esseen theorem due to Shevtsova \cite{shevtsova2011}.

Secondly, using the fact that $d_{TV}(\bin(m,p^\prime),\po(mp^\prime))\leq p^\prime$ for any parameters $m$ and $p^\prime$ (see page 8 of \cite{barbour}), we have that
\[
d_{TV}(G^\prime,G^{\prime\prime})\leq\ep[(1-p)^{M^\prime}]=e^{-Tpq_0}\,.
\]

To complete the proof, it remains only to bound $d_{TV}(G^{\prime\prime},Z)$. To that end, we apply Lemma \ref{lem:NBtoG}, noting that the parameters $q$ and $r$ of $Z$ are chosen such that the first two moments of $Z$ match those of $G^{\prime\prime}$: straightforward calculations show that $\ep[G^{\prime\prime}]=\mu$ and $\var(G^{\prime\prime})=\sigma^2$. Lemma \ref{lem:NBtoG} gives
\[
d_{TV}(G^{\prime\prime},Z)\leq\frac{2-q}{1-q}d_W(\xi,\eta)\,,
\]
where $\xi=(n-k)(1-p)^{M^\prime}$ and $\eta\sim\Gamma(r,\lambda)$ has a gamma distribution with rate parameter $\lambda=\frac{1-q}{q}$. Using scaling properties of the gamma distribution and the Wasserstein distance, we have that
\[
d_W(\xi,\eta)=d_W((n-k)\xi^\prime,(n-k)\eta^\prime)=(n-k)d_W(\xi^\prime,\eta^\prime)\,,
\]
where $\xi^\prime=(1-p)^{M^\prime}$ and $\eta^\prime\sim\Gamma(r,Kr)$, with $K$ as in the statement of the theorem. We then write
\begin{equation} \label{eq:finalterm}
d_W(\xi^\prime,\eta^\prime)=\int_0^1|\pr(\xi^\prime>x)-\pr(\eta^\prime>x)|\,dx+\int_1^\infty\pr(\eta^\prime>x)\,dx\,,
\end{equation}
and bound the two terms on the right-hand side of \eqref{eq:finalterm} separately. Beginning with the final term of \eqref{eq:finalterm}, 
%applying Cantelli's inequality gives
%\begin{multline*}
%\int_1^\infty\pr(\eta^\prime>x)\,dx=\int_{1-\ep\eta^\prime}^\infty\pr(\eta^\prime-\ep\eta^\prime>t)\,dt\l%eq\int_{1-\ep\eta^\prime}^\infty\frac{\var(\eta^\prime)}{\var(\eta^\prime)+t^2}\,dt\\
%=\frac{\sqrt{\var(\eta^\prime)}}{2}\left[\pi-2\arctan{\left(\frac{1-\ep\eta^\prime}{\sqrt{\var(\eta^\prim%e)}}\right)}\right]=\frac{1}{2 K\sqrt{r}}\left[\pi-2\arctan{\left(
%(K-1)\sqrt{r}\right)}\right]\,,
%\end{multline*} 
%where we note that it is straightforward to check that $\ep\eta^\prime=\frac{r}{r K}<1$. 
we note that, for $Z \sim \Gamma(\alpha,\beta)$, a standard Chernoff bounding argument gives us that, for any $z > \alpha/\beta$ and $t > 0$,
\begin{equation}
    \pr(Z > z) \leq \frac{\ep e^{tZ}}{e^{t z}} = \frac{ (1-t/\beta)^{-\alpha}}{e^{t z}} = \left( \frac{ \beta e}{\alpha} \right)^\alpha \exp(-\beta z) z^\alpha, \label{eq:chernoffbd}
\end{equation}
where we take the optimal choice that $t = \beta - \alpha/z$. (Observe that the same argument applies to bound 
\begin{equation}
    \pr(Z < z) \leq  \left( \frac{ \beta e}{\alpha} \right)^\alpha \exp(-\beta z) z^\alpha, \label{eq:chernoffbd2}
\end{equation}
for $z < \beta/\alpha$, simply by again taking $t = \beta - \alpha/z <0$).
Since $\eta^\prime \sim \Gamma(r, rK)$, the expression \eqref{eq:chernoffbd} tells us that 
$$ \pr( \eta^\prime > x) \leq (K e)^r \exp(-K r x) x^r.$$
This allows us to write the final term of \eqref{eq:finalterm} as
\begin{eqnarray} \int_1^\infty \pr( \eta^\prime > x) dx
& \leq & \frac{ e^r \Gamma(r+1)}{ r^{r+1} K}
\int_1^\infty \frac{ (K r)^{r+1}}{\Gamma(r+1)} x^r \exp(-K r x) dx \label{eq:gammaint} \\
& \leq & \frac{ e^r \Gamma(r+1)}{ r^{r+1} K}
\left( \frac{K r e}{r+1} \right)^{r+1} \exp(-K r) \nonumber \\
& = & \frac{e^{2r+1} \Gamma(r+1) K^r}{(r+1)^{r+1}} \exp(-K r) 
\leq e^{r+1} K^r \exp(-K r), \label{eq:gammaint2}
\end{eqnarray}
since we recognise the integrand in \eqref{eq:gammaint} as the density of a $\Gamma(r+1,K r)$ random variable and again apply \eqref{eq:chernoffbd}. The expression \eqref{eq:gammaint2} follows on observing that $v(r) := e^r \Gamma(r+1)/(r+1)^{r+1} \leq 1$ for $r \geq 0$. We can see this, for example, since $v(0) = 1$, and $v(r)$ is decreasing in $r$ since $\frac{d}{dr} \log v(r) = \psi(r+1) - \log(r+1) \leq 0$, where $\psi$ is the digamma function.

Finally, we write the first term of \eqref{eq:finalterm} as
\begin{align*}
\int_0^1|\pr(\xi^\prime>x)-\pr(\eta^\prime>x)|\,dx&=\int_0^1\left|\pr\left(M^\prime<\left\lceil\frac{\log(x)}{\log(1-p)}\right\rceil\right)-\pr(\eta^\prime>x)\right|\,dx  \\
&=\int_0^1\left|\widehat{\Gamma}\left(\left\lceil\frac{\log(x)}{\log(1-p)}\right\rceil,T q_0\right)-\widehat{\Gamma}\left(r,K r x\right)\right|\,dx\,.
\end{align*}
This completes the proof of the theorem.
\end{proof}

\section{Proof of Proposition \ref{prop:exptails}}
\label{app:proofexptails}

\begin{proof}%[Proof of Proposition \ref{prop:exptails}]
For any $u > 0$ we may write, using Markov's inequality,
\begin{align}
\pr(Z \geq g)
& \leq  \frac{ \ep \left[ e^{uZ} \right]}{e^{u g}} 
 = \exp \left( \log M_Z(u) - u g \right) \nonumber \\
& = \exp \left( r \log \left(\frac{q}{1-(1-q) e^{u}} \right) - u g \right),
\label{eq:tails}
\end{align}
where we use the fact that the moment generating function of the negative binomial distribution $\mbox{NB}(r,q)$ is $$\ep [e^{uZ} ] = M_{Z}(u) =  \left(\frac{q}{1-(1-q) e^{u}}\right)^{r}.$$
Direct calculation then gives that  the optimal value of $u$ to substitute is
$$ u^* = \log \left( \frac{g}{(g+r)(1-q)} \right),$$
(note that the assumption $g > \ep Z = r(1-q)/q$ ensures that $\frac{g}{(g+r)(1-q)} > 1$ so that $u^* > 0$ as required in \eqref{eq:tails}). The result follows on substitution in \eqref{eq:tails}, since this choice of $u = u^*$ makes $$ \frac{q}{1-(1-q) e^{u}} = \frac{q( g+r)}{r}.$$
\end{proof}
 
 \section*{Acknowledgements}
 
 The collaboration between Letian Yu and Oliver Johnson was supported by a Charles Kao bursary from the Chinese University of Hong Kong in the summer of 2021. We thank the Associate Editor and a referee for their helpful comments and suggestions.

%\bibliographystyle{abbrv}
%\bibliography{references}
\end{document}